\documentclass[12pt,regno]{amsart}
\usepackage{ifthen,algorithm,algorithmic}
\usepackage{amssymb,amsfonts,amsmath,latexsym,dsfont,amsthm}
\usepackage{tikz,qtree,pgflibraryplotmarks}
\usepackage{fullpage}
\usepackage{graphicx}
\usepackage{mathrsfs}
\usepackage{setspace}
\usepackage[english]{babel}

\graphicspath{{images/}{./}}

\usepackage{boxedminipage}

\usepackage{pgf} 
\usepackage{subfigure} 

\newcommand{\FrameboxA}[2][]{#2}
\newcommand{\Framebox}[1][]{\FrameboxA}

\renewcommand{\pmod}[1]{{\ifmmode\text{\rm\ (mod~$#1$)}\else\discretionary{}{}{\hbox{ }}\rm(mod~$#1$)\fi}}

\newcommand{\abs}[1]{\lvert#1\rvert}

\newcommand{\Z}{\mathbb{Z}}

\newcommand{\lcm}{\mathop{\rm lcm }}
\newcommand{\li}{\mathop{\rm li }}

\newtheorem{theorem}{Theorem}
\newtheorem{lemma}[theorem]{Lemma}
\newtheorem{proposition}[theorem]{Proposition}
\newtheorem{corollary}[theorem]{Corollary}
\newtheorem{definition}[theorem]{Definition}

\newtheorem{conjecture}[theorem]{Conjecture}

\title{The number of iterates of the Carmichael lambda function required to reach 1}
\author{Nick Harland}
\address{Department of Mathematics, University of British Columbia, Room 121, 1984 Mathematics Road, Vancouver, BC, V6T 1Z2, Canada}
\email{harlandn@math.ubc.ca}

\begin{document}
\begin{abstract}
The Carmichael lambda function $\lambda(n)$ is defined to be the smallest positive integer $m$ such that $a^m \equiv 1 \pmod{n}$ for all $(a,n)=1.$ $\lambda_k(n)$ is defined to be the $k$th iterate of $\lambda(n).$ Let $L(n)$ be the smallest  $k$ for which $\lambda_k(n)=1.$ It's easy to show that $L(n) \ll \log n.$ It's conjectured that $L(n)\asymp \log\log n,$ but previously it was not known to be $o(\log n)$ for almost all $n.$ We will show that $L(n) \ll (\log n)^{\delta}$ for almost all $n,$ for some $\delta <1.$ We will also show $L(n) \gg \log\log n$ for almost all $n$  and conjecture a normal order for $L(n).$ 
\end{abstract}
\maketitle

\section{Introduction}\label{Intro}

The Carmichael lambda function $\lambda(n)$ is defined to be the exponent of the multiplicative group $(\Z / n\Z)^{\times}.$ It can be computed using the identity $\lambda(\lcm\{a,b\})=\lcm\{\lambda(a),\lambda(b)\}$ and its values at prime powers. Those value are $\lambda(p^{k})=\phi(p^{k})=p^k-p^{k-1}$ for odd primes $p,$ and $\lambda(2)=1,\lambda(4)=2,$ and $\lambda(2^{k})=\phi(2^{k})/2=2^{k-2}$ for $k \ge 3$. The $k$--fold iterated Carmichael lambda function is defined recursively as follows.
$$\lambda_0(n)=n,~~\lambda_k(n)=\lambda(\lambda_{k-1}(n)), \text{ for } k \ge 1.$$ This paper is about some analytical properties of a related function.
\begin{definition}
Let $L(n)$ be the smallest non-negative integer $k$ such that $\lambda_k(n)=1.$
\end{definition}
Since $\lambda(n)$ is either even or $1,$ and $\lambda(n) \le n/2 $ for even $n,$ we easily see that $L(n)\le \lfloor\log n/\log 2+1\rfloor.$ By considering when $n$ is a power of $3$ we can note that $L(n)\ge 1+ (1/\log 3)\log n$ for infinitely many values of $n.$ As for upper bounds, Martin and Pomerance \cite{MP} gave a construction for which $L(n) < (1/\log 2 +o(1))\log\log n$ for infinitely many $n.$ Probabilistically, these examples have asymptotic density $0.$ It is conjectured that for a set of positive integers with asymptotic density $1$, that $L(n)\asymp \log\log n,$ however no previous results have shown $L(n)=o(\log n)$ for almost all $n.$

The Pratt tree for a prime $p$ is defined as follows. Let the root node be $p.$ Below $p$ are nodes labelled with the primes $q$ such that $q \mid p-1.$ The nodes below $q$ are primes dividing $q-1$ and so on until we are left with just $2.$ For example, if we want to take the prime $3691,$ the primes dividing $3690$ are $2,3,5$ and $41.$ The primes dividing $3-1$ is $2,$ dividing $5-1$ is 2 and dividing $41-1$ are $2$ and $5.$ Continuing we obtain the tree

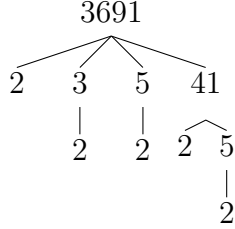
\Tree [.3691 2 [.3 2 ] [.5 2 ] [.41 2 [.5 2 ] ] ]

In a recent paper by Ford, Konyagin and Luca \cite{FKL}, they found bounds on the height of the Pratt tree $H(p).$ The height is closely related to $L(p)$ for a prime $p.$ It's easy to see that $H(p) \le L(p),$ so any lower bound on $H$ acts as a lower bound on $L.$ The Bombieri--Vinogradov Theorem implies 
\begin{equation}\label{BV}
\sum_{n\le Q}\max_{y\le x}\bigg|\pi(y;n,1)-\frac{\li(y)}{\phi(m)} \bigg| \ll x(\log x)^{-A}
\end{equation}
with $Q=x^{1/2}(\log x)^{-B}$ for any $A>0$ and $B=B(A).$ The Elliot--Halberstam conjecture says that \eqref{BV} holds for $Q=x^{\theta}$ for any $\theta<1.$ Let $\theta'$ be such that \eqref{BV} holds for $Q=x^{\theta'}.$ In \cite{FKL} they showed for any $c<1/(e^{-1}-\log \theta'),$
\begin{equation}\label{lower}
H(p)>c\log\log p
\end{equation} for all but $O\big( x/(\log x)^K\big)$ primes $p,$ for some $K>1.$ Bombieri--Vinogradov allows us to take any $c<1/(e^{-1}+\log 2),$ and under Elliot--Halberstam, we can take any $c<e.$

It's easy to see if $n=\prod p_i^{\alpha_i}$, then
\begin{equation}\label{primebreakdown}
L(n)=\max_i\{L(p_i^{\alpha_i})\}
\end{equation} and
\begin{equation}\label{powers}
L(p^{\alpha})=\alpha-1+L(p) \ge L(p).\end{equation} These two equations imply $L(n)\ge L(p)$ for any $p \mid n,$ motivating the following theorem.

\begin{theorem}\label{LowerBound} There exists some $c>0$ such that 
$$L(n) \ge c\log\log n$$ for all $n$ as $n \rightarrow \infty.$
\end{theorem}
For an upper bound, from \cite{FKL} we have 
\begin{equation}\label{upper}
H(p) \le (\log p)^{0.95022} 
\end{equation} for all $p\le x$ outside a set of size $O\big(x\exp({-}(\log x)^\delta)$ for some $\delta>0.$ We extend this to a result about $L(n).$

\begin{theorem}\label{UpperBound}
If $H(p) \le (\log p)^{\gamma}$ for almost all $p\le x$ outside a set of size $O\big(x\exp({-}(\log x)^\delta\big)$ for some $\delta>0,$ then for some function $\psi$ and any $\epsilon>0,$ $$L(n) \ll (\log n)^{\gamma}\psi(n)$$ for almost all $n$ as $n \rightarrow \infty.$
\end{theorem}The function $\psi(x)$ can be taken to be as small as $O(\log\log\log x).$ Using Theorem \ref{UpperBound} along with equation \eqref{upper} yields the following corollary.
\begin{corollary}
For any $\epsilon>0$, for almost all $n$, $$L(n) \ll (\log n)^{0.9503}.$$
\end{corollary} 

In \cite{FKL}, the authors described a probabilistic model which suggested a conjecture on the normal order for $H(p)$ is $e\log\log p.$ Assuming this conjecture, we give some evidence to suggest a related conjecture for $L(n).$

\begin{conjecture}\label{conj1}
The normal order of $L(n)$ is $e\log\log n.$
\end{conjecture} 
Throughout the paper, $p$ and $q$ will always denote primes, $\log_k(n)$ will denote the $k$th iterate of the $\log$ function, and $y=y(x)=\log_2(x)=\log\log x.$ Also the notation $q \prec q'$ is defined to mean  $q \mid q'-1.$

\section{Lower bound for $L(n)$}
For any $p \mid n$, we know that $L(n) \ge L(p),$ which implies that $L(n)>c\log_2(p).$ However, if all the primes $p$ dividing $n$ are small relative to $n,$ this will not imply that $L(n)>c\log_2(n).$ The proof of Theorem \ref{LowerBound} therefore relies on showing that not many $n$ are composed entirely of small primes as well as dealing with the exceptional set for which \eqref{lower} doesn't hold.

\begin{proof}[Proof of Theorem \ref{LowerBound}]

Let $Y=Y(x) \le x.$ Let $c$ from Equation \eqref{lower} which can be shown to be any constant $c < 1/(e^{-1}+\log 2).$ Define a set $S(x)=S(x,Y)=\{p : p \ge Y, H(p) < c\log_2(p)  \}.$ We have that $\# S(x) \ll x/(\log x)^K$ for some $K>1,$ so if $p \mid n$ for some $p \notin S(x),$ 

\begin{align*}
L(n) \ge L(p) \ge c\log_2(p).
\end{align*}
If $n$ is only composed of $p \in S(x)$, then either there exists $p \ge Y, p\in S(x)$ such that $p \mid n$ or $n$ is composed entirely of primes less than or equal to $Y.$ The number of $n \le x$ where there exists $p \mid n$ with $p \in S(x)$ is bounded by

\begin{align*}
\sum_{n\le x}\sum_{\substack{ p \mid n \\ p \in S(x)}} 1 &= \sum_{\substack{ p \le x \\ p \in S(x)}}\sum_{\substack{ n \le x \\ n \equiv 0 \pmod{p}}}1 \\ & \le  \sum_{\substack{ p \le x \\ p \in S(x)}}\frac{x}{p}
\\ & = x\bigg(\frac{\abs{S(x)}}{x} + \int_{Y}^{x}\frac{S(t)dt}{t^2}  \bigg)
\\ & \ll \frac{x}{\log^K x} + \int_{Y}^{x} \frac{dt}{t\log ^K t}
\\ & \ll \frac{x}{\log^{K-1}Y}
\end{align*} using partial summation. Let $\Psi(x,z)$ be the number of $n \le x$ composed of primes $p \le z$ and let $z=x^{1/u}.$  By \cite[Theorem 7.2]{MV}, $$\Psi(x,z) \ll x\rho(u)$$ where $\rho(u)$ is the Dickman function. It's known that $\rho(u)\rightarrow 0$ as $u \rightarrow \infty.$ Given $\epsilon > 0,$ choose $Y$ such that $\log Y = (\log x)^{1-\epsilon}.$ Since $Y<x^{\gamma}$ for all $\gamma>0,$ this choice yields
$L(n)\ge c(1-\epsilon)\log_2(x)$ for all but $O\big(x/(\log Y)^{K-1} + \Psi(x,Y)\big)=o(x)$ such $n.$ This completes the theorem.

\end{proof}

It's worth noting that under the Elliot--Halberstam conjecture, that constant can be replaced by any $c < e.$

\section{Upper Bound for $L(n)$}

The Pratt tree for a prime $p$ describes the primes $q$ where $q \prec \dots \prec p.$ This is useful in calculating $L(p),$ however $L(p)$ is also increased by prime powers for which the Pratt tree does not describe. The proof of Theorem \ref{UpperBound} hinges on bounding the contribution of these large prime powers. We begin with the following lemma.

\begin{lemma}\label{cases}
Fix a prime $q$ and positive integers $k,\alpha.$ The number of $n \le x$ such that there exists $q^{\alpha} \mid q_{k-1}-1,q_{k-1} \mid q_{k-2}-1,\dots,q_1 \mid p-1$ and $p \mid n$ is at most $$\frac{x(cy)^k}{q^{\alpha}}$$ for some absolute constant $c.$
\end{lemma}
\begin{proof}
We'll make use out of the Brun--Titchmarsh inequality
$$\pi(t;m,a)\le \frac{2t}{\phi(m)\log(t/m)}.$$ Partial summation yields
\begin{equation}\label{BT}
\sum_{\substack{p \le x \\ p \equiv 1 \pmod{m}}}\frac{1}{p} \ll \frac{\log_2 x}{\phi(m)}.
\end{equation} Noting that $\phi(m) \gg m$ if $m$ is a prime or prime power implies
\begin{equation}\label{BT2}
\sum_{\substack{p \le x \\ p \equiv 1 \pmod{m}}}\frac{1}{p} \le \frac{c\log_2 x}{m} = \frac{cy}{m}
\end{equation} if $m$ is a prime or prime power. Repeated uses of \eqref{BT2} gives us the number of such $n$ is bounded by
\begin{align*}
\sum_{n \le x}\sum_{p \mid n}& \sum_{q_1 \mid p-1}\dots \sum_{q_{k-1} \mid q_{k-2}-1}\sum_{q^{\alpha} \mid q_{k-1}-1}1 \\ & = \sum_{q_{k-1}\equiv 1\pmod{q^{\alpha}}}\sum_{q_{k-2}\equiv 1\pmod{q_{k-1}}}\dots\sum_{p \equiv 1 \pmod{q_1}}\sum_{\substack{n \le x \\ n \equiv 0 \pmod{p}}}1 \\
& \le \sum_{q_{k-1}\equiv 1\pmod{q^{\alpha}}}\sum_{q_{k-2}\equiv 1\pmod{q_{k-1}}}\dots\sum_{p \equiv 1 \pmod{q_1}}\frac{x}{p} \\
& \le \sum_{q_{k-1}\equiv 1\pmod{q^{\alpha}}}\sum_{q_{k-2}\equiv 1\pmod{q_{k-1}}}\dots\sum_{q_1 \equiv 1 \pmod{q_2}}\frac{xcy}{q_1} \\
& \le \sum_{q_{k-1}\equiv 1\pmod{q^{\alpha}}}\frac{x(cy)^{k-1}}{q_{k-1}} \\
& \le \frac{x(cy)^{k}}{q^{\alpha}}.
\end{align*}
\end{proof}

We will show Theorem \ref{UpperBound} is a corollary to the main propostion, that the difference between $H(p)$ and $L(p)$ cannot be too great.

\begin{proposition}\label{Prop}Let $b>0$ and $c$ be the constant from \eqref{BT}. Suppose $H(p) \le (\log p)^{\gamma}$ for all $p\le x$ outside a set of size $O\big(x\exp({-}(\log x)^\delta\big)$ and let $\psi(x)$ be a function such that
\begin{equation}
\frac{x(cy)^{(\log x)^\gamma+1}}{2^{b(\log x)^\gamma\psi(x)-2}}=o(x).
\end{equation} Then
$$L(n)\ll (\log x)^\gamma\psi(x)$$ for almost all $n \le x,$ for which the excluded $n$ are divisible by at least one prime $p$ in the above excluded set.
\end{proposition}
Note that if $\psi'(x)$ is some function such that $b\psi'(\log x)^{\gamma}-\log(cy)\rightarrow \infty$ and $\psi(x)>\frac{1}{b\log 2}\log(cy)+\psi'(x),$ then 
\begin{align*}\frac{x(cy)^{(\log x)^\gamma+1}}{2^{b(\log x)^\gamma\psi(x)-2}} & = \frac{x\exp\big(((\log x)^{\gamma}+1)\log(cy)\big)}{\exp\big((b\psi(x)(\log x)^\gamma-2)\log 2\big)}\ll x\exp\bigg(\log(cy)-b\psi'(x)(\log x)^{\gamma}  \bigg)=o(x).
\end{align*} Specifically we can choose $\psi(x)\ll_b \log_3(x).$ The proof of Propostion \ref{Prop} begins by analyzing the ways that $L(p)$ can be much larger than $H(p)$ and then showing in those cases that it cannot happen for many $p.$ 

\begin{proof}[Proof of Proposition \ref{Prop}]
Let $n=\prod p_i^{\alpha_i}$ be the prime factorization of $n$ where $p\mid n \rightarrow H(p) \le (\log p)^{\gamma}.$ By equations \eqref{primebreakdown} and \eqref{powers}, $L(n)=\max_i\{\alpha_i-1+L(p)\}$. Our first goal is to show that the number of $n$ for which there exists a large $\alpha$ with $p^\alpha \mid n$ is small. Fixing a prime $p$, the number $n \le x$ such that $p^\alpha \mid n$ is at most $x/p^\alpha.$ Hence the number of bad $n$ is bounded by
\begin{align*}\sum_{p\le x}\frac{x}{p^\alpha} &\le x\sum_{m=1}^x\frac{1}{p^\alpha} \\
& \ll \frac{x}{\alpha}.
\end{align*} Applying this with any $\alpha= \xi(x)$ with $\xi(x) \rightarrow \infty$ makes the number of such $n$ be $o(x).$ Therefore for almost all $n \le x$ we can assume $$L(n) \le \max_{p \mid n}(L(p)+\xi(x)) = \max_{p \mid n}(L(p)+o((\log x)^\gamma)$$ by taking $\xi(x)=o((\log x)^\gamma).$

Let $\psi(x)$ be a function satifying the hypothesis of the proposition. We must determine how $L(p)$ can be larger than $H(p)$ and by how much. First note that for any prime in the Pratt tree, the difference between the factors of $q-1$ and the primes in the Pratt tree are just the powers of that prime which divide $q-1.$ Therefore, if we have a branch of the Pratt tree, $2=q_k \prec q_{k-1} \prec \dots \prec q_1 \prec q_0 = p,$ then $L(p)\le \max\{H(p) + \sum_{i=1}^{k}(\alpha_i-1)\}$ where $q_i^{\alpha_i} \| q_{i-1}-1$ and the $\max$ is taken over all the branches of the Pratt tree. The inequality $q_i^{\alpha_i} < q_{i-1}$ holds for all $i$ which implies $$2^{\prod_{i=1}^k \alpha_i} < p.$$ Therefore we need to maximize the sum $\sum_{i=1}^{k}(\alpha_i-1)$ subject to $\prod_{i=1}^k \alpha_i < \log x/\log 2.$ 

Suppose we have $rs=tu,$ where $2 \le r,s,t,u \le M.$ The larger of $r+s$ and $t+u$ will be where the two terms are further apart. Consequently if we wish to maximize a sum subject a fixed product and number of terms, we want some terms to be the lowest possible value, in this case 2, and the rest to be the largest value, in this case $M.$ Suppose $\sum_{i=1}^{k}(\alpha_i-1) \gg \psi(x)(\log x)^\gamma,$ where $2 \le \alpha \le M$ and $M=o\big(\psi(x)(\log x)^\gamma\big).$ By the above reasoning we know the sum is bounded by $2(k-l)+lM$ for some $l\le k.$ However, $M^l \le \log x/\log 2$ implying $l \le (\log_2 x-\log_2 2)/\log M.$ Since $k \ll \log_2(x),$ $2(k-l)+lM$ is bounded above by 
$$O\bigg(\log_2(x)+M(\log_2 x-\log_2 2)/\log M \bigg)=o\big(\psi(x)(\log x)^\gamma\big),$$ contradicting the bound on $M.$ As a result, we know there exists some $\alpha_i\ge  b\psi(x)(\log x)^\gamma$ for some $b>0.$

It remains to show that the number of $n \le x$ such that there exists $q^{\alpha} \mid q_{k-1}-1,q_{k-1} \mid q_{k-2}-1,\dots,q_1 \mid p-1,p \mid n$, with $\alpha \ge b\psi(x)(\log x)^\gamma$ is $o(x).$ Note that $k \le H(p) \le (\log x)^\gamma.$ By Lemma \ref{cases}, the number of $n$ is bounded by

\begin{align*}
\sum_{\alpha \ge b\psi(x)(\log x)^\gamma}\sum_{k \le (\log x)^\gamma}\sum_q\frac{x(cy)^k}{q^\alpha}.
\end{align*} Summing q over all integers at least 2 instead of primes and using $\alpha \ge 2$ makes this
\begin{align*}
\ll \sum_{\alpha \ge b\psi(x)(\log x)^\gamma}\sum_{k \le (\log x)^\gamma}\frac{x(cy)^k}{2^{\alpha-1}}.
\end{align*} Summing the geometric series under both $\alpha$ and $k$ yields
\begin{align*}
\ll \frac{x(cy)^{(\log x)^\gamma+1}}{2^{b\psi(x)(\log x)^\gamma-2}}.
\end{align*} By the choice of $\psi$ this is $o(x)$ and hence for almost all $n \le x,$
$$L(n) \le o((\log x)^\gamma) + \max_{p \mid n} \bigg\{H(p) + \sum_{i=1}^{k}(\alpha_i-1)\bigg\} \ll (\log n)^\gamma + \psi(x)(\log x)^\gamma \ll \psi(x)(\log x)^\gamma.$$
\end{proof}

We are now in a position to prove Theorem \ref{UpperBound}. Proposition \ref{Prop} yields the theorem provided $n$ wasn't divisible by any primes for which \eqref{upper} fails to hold, so it remains to consider when $n$ is divisible by such a prime.

\begin{proof}[Proof of Theorem \ref{UpperBound}]
Let $Y=Y(x)\rightarrow \infty$ such that $\log Y \ll (\log x)^{\gamma}.$ As in the proof of Theorem \ref{LowerBound} we know that the set of $n \le x$ which are composed entirely of primes less than or equal to $Y$ has density 0. Therefore we only need to consider values of $n$ for which there exists a prime greater than $Y$ where $H(p)>(\log p)^{\gamma}.$ Let $S(x)$ be the set $\{Y < p \le x \mid L(p)>(\log p)^{\gamma} \}.$ Since $L(p)>H(p),$ by \eqref{upper} we know that $\#S(x) \ll x\exp\big({-}(\log t)^\delta\big).$ The number of $n \le x$ where $n$ is divisible by a prime in $S(x)$ is bounded by
\begin{align*}
\sum_{n \le x}\sum_{\substack{p \in S(x) \\ p \mid n}} 1 &\le \sum_{p \in S(x)}\frac{x}{p} \\
& = \frac{x\abs{S(x)}}{x}+x\int_Y^x\frac{\abs{S(t)}dt}{t^2} \\
& \ll x\exp\big({-}(\log x)^\delta \big)+x\int_Y^x\frac{\exp\big({-}(\log t)^\delta\big) dt}{t} \\
& \ll x\exp\big({-}(\log x)^\delta \big)+\frac{x}{\log x}+\frac{x}{\log Y}
\end{align*}using partial summation. In the last line we used $\exp({-}(\log t)^{\delta}) \ll (\log t)^{-2}.$ By our choice of $Y$ the number of $n$ is $o(x)$ completing the theorem.
\end{proof}

\section{Conjecture for the normal order of $L(n).$}
The purpose of this section is to justify Conjecture \ref{conj1} assuming the conjecture in \cite{FKL} which implies $H(p)\le e\log_2 p$ for almost all $p.$ To do this, we wish to analyze the difference $L(p)-H(p)$ to show that it is not too large. As we saw in the previous section, this difference is created when a branch of the Pratt tree has $p_i^a \mid p_{i-1}-1$ where $a>1.$ Let $Y=Y(x)\le x.$ Also let a branch of the Pratt tree be $p_1\succ p_2 \succ \dots \succ p_l \succ p_{l+1} \succ \dots \succ p_k=2$ where $p_i^{a_i} \| p_{i-1}-1$ and let $l$ be the largest index such that $p_l>Y.$ We will separate our arguments into the cases where  $i<l+1, i>l+1,$ and finally $i=l+1.$

By the trivial estimate $L(n) \ll \log n$ we know $L(p_{l+1})\ll \log Y.$ By a suitable choice of $Y$ this will be made to be $o(\log_2 x).$

For $i \le l,$ we wish to know the probability that $n$ has a factor $p^a,$ where $p>Y.$ We use the following lemma.

\begin{lemma}\label{Prob}
The number of $n\le x$ for which there exists $p>Y$ where $p^a \| n$ is $O(x/Y^{a-1}).$
\end{lemma}

\begin{proof}
The number of $n$ is bounded by
\begin{align*}
\sum_{n\le x}\sum_{\substack{p>Y \\ p \mid n}}1 \le \sum_{p>Y}\frac{x}{p^a}\ll \frac{x}{Y^{a-1}}.
\end{align*}
\end{proof}

By Lemma \ref{Prob} we should expect a proportion of at most $c/Y^a.$ This implies that the probability of $p_i^{a_i} \| p_{i-1}-1$ where $(a_2-1)+(a_3-1)+\dots+(a_l-1)=\psi(x)$ is bounded by $c^{l}/Y^{\psi(x)}.$ Since the number of possible branches of the Pratt tree is trivially bounded by $\log x,$ the probability of there existing such a string of $a_i$ is bounded by
$$1-\bigg(1-\frac{c^l}{Y^{\psi(x)}}\bigg)^{\log x}.$$ This bound will approach $0$ provided $\log x =o\big(Y^{\psi(x)}/c^l \big).$ Under the assumption that $H(p)\le e\log_2(p),$ we have $l \le H(p) \le e\log_2(p).$ Therefore a choice of  $Y=\exp((\log_2(x))^{3/4})$ and $\psi(x)=(\log_2(x))^{3/4}$ makes the contribution to $L(p)-H(p)$ be $o(\log_2(x))$ for $i\ne l+1.$

For $i=l+1,$ we have $p_{l+1}^{a_{l+1}} \mid p_l-1.$ The remaining contribution to $L(p)-H(p)$ is $a_{l+1}-1,$ if $p_{l+1}>2$ and $\lceil (a_{l+1}+1)/2\rceil$ if $p_{l+1}=2.$ For the $a_{l+1}$ to contribute a lot to $L(p),$ it must be at the end of a long prime chain, i.e. $l \gg \log_2 p,$ otherwise the conjectured value of $H(p)$ being $e\log_2 p$ would nullify the contribution. To show this is unlikely, we use a result from \cite{BKW} which implies that the number of primes at a {\em fixed} level $n$ of the Pratt tree is $\sim (\log_2 p)^n/n!.$ If we allow some dependence and use $n=c\log_2 p,$ for $0<c<\log_2 p$ we get roughly $(e/c)^{c\log_2 p}=(\log p)^{c\log(e/c)}$ primes at level $n.$ We show that the probability of none of these primes being congruent to $1$ modulo $p_{l+1}^{a_{l+1}}$ goes to $1$ provided $p_{l+1}^{a_{l+1}}$ is large enough.

Suppose we have $N$ primes. The probability that any one of them is congruent to $1$ modulo $r^a$ for a prime $r$ and positive integer $a,$ is $1/\phi(r^a).$ Assuming some independence, the probability that none of the $N$ primes are congruent to $1$ modulo $r^a$ is 
$$\bigg(1-\frac{1}{\phi(r^a)}  \bigg)^N.$$ Let $\psi$ be a function going to infintiy. Furthermore, let $r^a > N\psi(N),$ be a prime power. Since $r$ is prime, we know $\phi(r^a)\ge r^a/2.$ This bound implies the probability is bounded below by 
$$\bigg(1-\frac{2}{r^a}  \bigg)^N.$$ Using our lower bound on $r^a$ we get
$$\bigg(1-\frac{2}{r^a}  \bigg)^N \ge 1-\bigg(1-\frac{2}{N\psi(N)}  \bigg)^N \rightarrow 1.$$ 

We know wish to use the lower bound on $r^a$ to bound $a_{l+1}$ and therefore our contribution to $L(p)-H(p).$ Suppose $q_{l+1}\ne 2.$ If the level $l \approx c\log_2 p,$ for almost all $p,$ we expect $$a_{l+1}\le \frac{\log(N\log N)}{\log q_{l+1}} = \frac{c\log(e/c)}{\log q_{l+1}}\log_2 p + O(\log_3 p).$$ Combining all the contributions along any particular branch, we get \begin{equation}\label{Jamie} L(p) \le \bigg(c+\frac{c\log(e/c)}{\log q_{l+1}}\bigg)\log_2 p + o(\log _2 p).\end{equation} If $q_{l+1}=2,$ since,  $\lambda(2^a)=2^{a-2}$ we get
$$\bigg(c+\frac{c\log(e/c)}{2\log 2 }\bigg)\log_2 p + o(\log _2 p) = \bigg(c+\frac{c\log(e/c)}{\log 4 }\bigg)\log_2 p + o(\log _2 p).$$ Consequently, 3 is the value of 
$q_{l+1}$ which yields the largest coefficient of $\log_2 p$ in \eqref{Jamie}. Since $c+c\log(e/c)/\log 3 \le e$  for $0<c<e,$ we conclude that for almost all $p\le x,$ $L(p) \sim e\log_2 p.$ The reason that we can interchange $p$ and $n$ is the same reason as in Theorem \ref{LowerBound}.

It may seem obvious to conclude $L(p)\sim e\log_2 p$, since $H(p) \sim e\log_2 p.$ However, note that the function $\big(c+\frac{c\log(e/c)}{\log 2}\big)$ does not yield a maximum value of $e,$ but instead has its maximum of $2/\log(2)$ at $c=2.$ This may suggest if we had a function $L'(n)$ similar to $L(n)$ except that $\lambda'(2^a)=2^{a-1}$ for all positive integers $a,$ that we may get a different normal order, perhaps even $2\log_2 n/\log(2).$

\section*{Acknowledgements}
The author would like to thank Greg Martin for his guidance and helpful suggestions regarding Conjecture \ref{conj1}.

\end{document}